\documentclass[11pt, reqno]{article}
\usepackage{amsmath,amssymb,amsfonts,amsthm}
\usepackage{color}

\usepackage[colorlinks=true,linkcolor=blue,citecolor=blue,urlcolor=blue,pdfborder={0 0 0}]{hyperref}
\usepackage{cleveref}

\usepackage{caption}
\usepackage{thmtools}

\usepackage{mathrsfs}

\crefname{theorem}{Theorem}{Theorems}
\crefname{thm}{Theorem}{Theorems}
\crefname{lemma}{Lemma}{Lemmas}
\crefname{lem}{Lemma}{Lemmas}
\crefname{remark}{Remark}{Remarks}
\crefname{prop}{Proposition}{Propositions}
\crefname{defn}{Definition}{Definitions}
\crefname{corollary}{Corollary}{Corollaries}
\crefname{conjecture}{Conjecture}{Conjectures}
\crefname{question}{Question}{Questions}
\crefname{chapter}{Chapter}{Chapters}
\crefname{section}{Section}{Sections}
\crefname{figure}{Figure}{Figures}

\theoremstyle{plain}
\newtheorem{thm}{Theorem}[section]

\newtheorem{lemma}[thm]{Lemma}

\newtheorem{corollary}[thm]{Corollary}

\newtheorem{prop}[thm]{Proposition}

\theoremstyle{definition}

\theoremstyle{remark}
\newtheorem{remark}[thm]{Remark}

\numberwithin{equation}{section}

\newcommand{\Z}{\mathbb Z}

\usepackage[margin=0.95in]{geometry}
\usepackage{extarrows}
\usepackage{commath}
\usepackage{mathtools}
\newcommand{\eps}{\varepsilon}
\usepackage{bbm}
\usepackage{setspace}
\usepackage{enumitem}
\usepackage{tikz}
\usepackage{relsize}

\title{The Hammersley-Welsh bound for self-avoiding walk revisited}

\author{Tom Hutchcroft}

\begin{document}

\maketitle

\begin{abstract}
The Hammersley-Welsh bound (\emph{Quart.\ J.\ Math.,} 1962) states that the number $c_n$ of length $n$ self-avoiding walks on $\Z^d$ satisfies 
\begin{align*}
c_n &\leq 
 \exp \left[ O(n^{1/2}) \right] \mu_c^n,
\end{align*}
where $\mu_c=\mu_c(d)$ is the connective constant of $\Z^d$. While  stronger estimates have subsequently been proven for $d\geq 3$, for $d=2$ this has remained the best rigorous, unconditional bound available. In this note, we give a new, simplified proof of this bound, which does not rely on the combinatorial analysis of unfolding.  We also prove a small, non-quantitative improvement to the bound, namely
\begin{align*}
c_n &\leq \exp\left[ o(n^{1/2})\right] \mu_c^n.
\end{align*}
The improved bound is obtained as a corollary to the sub-ballisticity theorem of Duminil-Copin and Hammond (\emph{Commun.\ Math.\ Phys.,} 2013). We also show that any quantitative form of that theorem would yield a corresponding quantitative improvement to the Hammersley-Welsh bound.  
\end{abstract}

\section{Introduction}

Fix $d\geq 2$ and consider the hypercubic lattice $\Z^d$. A \textbf{self-avoiding walk} (SAW) is a simple path in $\Z^d$, that is, a path that does not visit any vertex more than once.  Self-avoiding walk was introduced as a model of a linear polymer in a good solvent by Flory and Orr \cite{flory1953principles,orr1947statistical}. The rigorous study of self-avoiding walk leads to many questions that are easy to state but difficult to solve, several of which are still open.  See \cite{MR2986656,MR3025395} for detailed introductions to the theory. 

Write $\Omega$ for the set of self-avoiding walks, and let $c_n$ be the number of length-$n$ self-avoiding walks starting at the origin.
Hammersley and Morton \cite{MR0064475} observed that the sequence $(c_n)_{n\geq0}$ is \textbf{submultiplicative}, meaning that $c_{n+m}\leq c_n c_m$ for every $n,m\geq 0$. It follows by Fekete's Lemma \cite{MR1544613} that there exists a constant $\mu_c=\mu_c(d)$, known as the \textbf{connective constant} of $\Z^d$, such that
\[ \mu_c^n \leq c_n \leq \mu_c^{n+o(n)}\]
for every $n\geq 0$.
Submultiplicativity arguments alone do not give any control of the subexponential correction to the growth of $c_n$, and it is a major open problem to determine the true asymptotics of $c_n$ when $d=2,3,4$.

It is believed that in fact the number of self-avoiding walks satisfies 
\begin{align*}
c_n \approx \begin{cases}
 n^{\gamma_2-1} \mu_c^n & d=2\\
 n^{\gamma_3-1} \mu_c^n & d=3\\
 (\log n)^{1/4} \mu_c^n & d=4\\
 \mu_c^n & d\geq 5.
\end{cases} && n\to \infty 
\end{align*}
for some constants $\gamma_2$ and $\gamma_3$.  For $d\geq 5$ this conjecture was verified in the seminal work of Hara and Slade \cite{MR1171762,MR1174248}. Rapid progress is being made on the four-dimensional case of the conjecture, including most notably a proof of the analogous conjecture for four-dimensional \emph{weakly} self-avoiding walk by Bauerschmidt, Brydges, and Slade \cite{MR3339164}. For $d=2,3$ the conjecture is wide open.

 The gap between what is conjectured and what is known for $d=2$ is very large. 
 Nienhuis \cite{MR675241} used non-rigorous Coulomb gas methods to compute that $\gamma_2=43/32$. This conjectured value of $\gamma_2$ is strongly supported by numerical evidence \cite{MR1342245,MR2065628}, non-rigorous conformal field theory arguments \cite{MR947005,MR1103831}, and by the theory of SLE \cite{MR2112127} (see also \cite{1608.00956}). In spite of all this, the best rigorous, unconditional estimate on $c_n$ for $d=2$  was until now the following theorem of Hammersley and Welsh \cite{MR0139535}. For $d\geq 3$, a similar stretched exponential bound with a better exponent was proven by Kesten~\cite{MR0166845}, see also  \cite[Section 3.3]{MR2986656}.


\begin{thm}[Hammersley-Welsh]
Let $d\geq2$. Then
\[
c_n \leq \exp \left[ \sqrt{\frac{2\pi^2 n}{3}} + o(n^{1/2}) \right] \mu_c^n = \exp \left[ O(n^{1/2}) \right] \mu_c^n 
\]
as $n\to\infty$.
\end{thm}

In this note, we prove the following slight improvement to the Hammersley-Welsh bound, which we show to be a corollary  to the work of Duminil-Copin and Hammond \cite{MR3117515} on sub-ballisticity of the self-avoiding walk.

\begin{thm}
\label{thm:main}
Let $d\geq 2$. Then 
$
c_n \leq  \exp \left[ o(n^{1/2}) \right] \mu_c^n$ 
as $n\to\infty$.
\end{thm}

Along the way, we also present a simplified proof of the Hammersley-Welsh bound (with a suboptimal constant in the exponent) that does not rely on either the combinatorial analysis of `unfolding' or the analysis of integer partitions.  We believe that both have been used in all published proofs of Hammersley-Welsh to date. We also believe that this simplified method of proof will be useful for obtaining further improvements to the bound in the future.

Let us briefly discuss the theorem of Duminil-Copin and Hammond. We will not in fact use their main result, but rather an intermediate result of theirs concerning self-avoiding \emph{bridges} \cite[Corollary 2.4 and Theorem 2.5]{MR3117515}. 
A self-avoiding walk $\omega$ is a \textbf{self-avoiding bridge} (SAB) if the $d$th coordinate of $\omega$ is uniquely minimized by its starting point, and is maximized (not necessarily uniquely) by its endpoint. Let $b_n$ be the number of self-avoiding bridges of length $n$ starting at the origin. Probabilistically, \cref{thm:DCH} states that for every fixed $\eps>0$, a uniformly chosen length $n$ SAB is exponentially unlikely to reach height $\eps n$, where the rate of exponential decay depends on $\eps$ (in some unknown way) but is always positive. Note that a length $n$ SAB cannot reach height greater than $n$. We write $\omega:v \to A$ to mean that $\omega$ starts at the vertex $v$ and ends in the set $A$.

\begin{thm}[Duminil-Copin and Hammond]
\label{thm:DCH}
Let $d\geq 2$. There exists an increasing%
\footnote{In \cite{MR3117515} it is not stated that the function is increasing. However, if $\phi:(0,1]\to (0,\infty)$ satisfies \eqref{eq:phidef}, then $\tilde \phi(\eps) = \sup \{\phi(\delta) : 0<\delta \leq \eps\}$ is positive, increasing and also satisfies \eqref{eq:phidef}.}
 function $\phi=\phi_d:(0,1]\to(0,\infty)$ such that
\begin{align}
\label{eq:phidef}
\sum_{\omega \in \Omega}\mathbbm{1}\bigl[\omega: 0 \to \Z^{d-1} \times \{m,m+1,\ldots\} \text{\emph{ is a length $n$ SAB}}\bigr]   \leq   b_n \exp\left[-\phi\left(\frac{m}{n}\right)n \right] 
\end{align}
for all $n \geq m >0$.
\end{thm}

The proof of \cref{thm:DCH} is not quantitative and does not give any estimates on the function $\phi$. However, our derivation of \cref{thm:main} from \cref{thm:DCH} \emph{is} quantitative, so that any quantitative version of \cref{thm:DCH} would yield a quantitative improvement to Hammersley-Welsh. Indeed, suppose that $\phi:(0,1]\to[0,\infty)$ satisfies \eqref{eq:phidef}. Define $\Phi:(0,1]\to[0,\infty)$ by
\[
\Phi(\eps) = \inf_{\eps \leq \delta \leq 1} \delta^{-1}\phi(\delta)
\]
 and define $\psi:(0,1)\to[1,\infty)$ by
\[\psi(\eps)=\sup\left\{\lambda \geq 1 : \eps \leq 1- \exp\left[ - \frac{\Phi(\lambda^{-1})}{\lambda-1} \right]\right\}. \]
If $\phi$ is positive and increasing  then $\Phi$ is positive and increasing also, and it follows that $\psi(\eps) \to\infty$ as $\eps\to 0$. 
Define $\Psi:\{0,1,\ldots\}\to [0,\infty)$ by
\[
\Psi(n) = \inf_{0<\eps< 1}\left[2
\left[1-(1-\eps)^{\psi(\eps)}\right]^{-1} 
- (n+1) \log(1-\eps) \right].
\]
The following is a quantitative version of \cref{thm:main}.
\begin{thm} 
\label{thm:quant}
Let $d\geq 2$. Suppose that $\phi:(0,1]\to[0,\infty)$ is increasing and satisfies \eqref{eq:phidef}, and let $\Psi$ be as above. Then
\[ c_n \leq    \exp \left[\Psi(n) -2\right] \mu_c^{n+1} \]
for every $n\geq 0$.
\end{thm}

To deduce \cref{thm:main} from \cref{thm:quant} it suffices to show that $\Psi(n)=o(n^{1/2})$ when $\phi$ is positive and increasing, which is straightforward. The classical Hammersley-Welsh bound (or rather the variant stated in Proposition \ref{prop:weakHW} below) can also be deduced from \cref{thm:quant} and the observation that $\phi\equiv 0$ trivially satisfies \eqref{eq:phidef}.  

Finally, we remark 
that any polynomial estimate on the function $\phi$ appearing in \cref{thm:DCH} would yield an improved exponent in the Hammerlsey-Welsh bound.  We do not treat values $\nu \leq 1$ as these cannot possibly occur; see Remark \ref{remark}.

\begin{corollary}
\label{cor:poly}
Let $d\geq 2$, and suppose that $\phi(\eps)=C \eps^{\nu}$ satisfies \eqref{eq:phidef} for some $\nu>1$ and $C>0$. Then 
\begin{equation*}
c_n \leq \exp\left[ O\bigl(n^{(\nu-1)/(2\nu-1)}\bigr) \right] \mu_c^n.
\end{equation*}
\end{corollary}


\begin{remark}
The central step in the proof of \cref{thm:quant} is to show that the generating function $B(z)=\sum_{n=0}^\infty z^n b_n$ satisfies
\[
B((1-\eps)z_c) \leq \left[ 1- (1-\eps)^{\psi(\eps)}\right]^{-1}.
\]
In particular, we obtain that if $\phi(\eps)=C\eps^\nu$ satisfies \eqref{eq:phidef} for some $C>0$ and $\nu>1$, then 
\begin{equation}
B((1-\eps)z_c) =O\left( \eps^{(1-\nu)/\nu} \right) \quad \text{as }\eps\to 0. 
\end{equation}
This can be thought of as an inequality between the `sub-ballisticity exponent' and the bridge counting exponent. 
\end{remark}

\section{Proof}

\subsection{Proof of Hammersley-Welsh}

Fix $d\geq 2$.
In this section we give our new proof of the classical Hammersley-Welsh bound. Our starting point will be the following inequality between generating functions due to Madras and Slade \cite[eq.\ 3.1.13]{MR2986656}.
 We define
\[\chi(z) = \sum_{n\geq 0} z^n c_n \quad  \text{ and } \quad
B(z)=\sum_{n\geq0} z^n b_n
\]
to be the generating functions of self-avoiding walks and self-avoiding bridges respectively. We define $z_c=\mu_c^{-1}$, which is the radius of convergence of $\chi(z)$ and hence also of $B(z)$ by the following proposition.

\begin{prop}[Madras and Slade]
\label{prop:M&S}
$\chi(z) \leq z^{-1}\exp\left[ 2B(z)-2\right]$ for every $z\geq 0$.
\end{prop}

This inequality relies on similar ideas as the proof of the Hammersley-Welsh bound, but is easier to prove. It does not rely on the combinatorial analysis of the `unfolding' of walks. We will prove that the following version of Hammersley-Welsh with a suboptimal constant can be deduced directly from Proposition \ref{prop:M&S} by elementary methods.

\begin{prop}
\label{prop:weakHW}
$c_n \leq \exp\left[ \sqrt{8n} + O\left(\sqrt{1/n}\right) \right] \mu_c^{n+1}.$
\end{prop}

Let $L_n = \Z^{d-1} \times \{n\}$ for each $n\geq 0$. 
For each $z\geq 0$ and $n\geq 0$, define
\[
a(z;n) = \sum_{\omega \in \Omega} z^{|\omega|} \mathbbm{1}\bigl[ \omega: 0 \to  L_n \text{ is a SAB}\bigr].
\]
If $\omega_1:0\to L_n$ and $\omega_2:0\to L_m$ are bridges, then we can form a bridge $\omega:0\to L_{n+m}$ by applying a translation to $\omega_2$ so that it starts at the endpoint of $\omega_1$ and then concatenating the two paths. This implies that the sequence $a(z;n)$ is supermultiplicative for each $z\geq 0$, meaning that
\[a(z;n+m)\geq a(z;n)a(z;m)\]
for every $n,m\geq 0$. 
It follows by Fekete's lemma that for each $z\geq 0$ there exists $\xi(z)$ such that
\begin{equation}
\label{eq:xidef}
\xi(z) = \lim_{n\to\infty} - \frac{1}{n} \log a(z;n) \in [-\infty,\infty]
\end{equation}
and that
\begin{equation}
\label{eq:abound}
a(z;n) \leq e^{-\xi(z)n}\end{equation}
for every $n\geq 0$. 

\begin{lemma}
\label{lem:xizc}
$\xi(z_c) \geq 0$. 
\end{lemma}

\begin{proof}
For each $n\geq 0$,  $a(z;n)$ is expressible as a power series in $z$ with non-negative coefficients, and is therefore left-continuous in $z$ for $z>0$.
If $z<z_c$ then $B(z)=\sum_{n\geq 0}a(z;n)<\infty$, so that $\xi(z)\geq 0$ by the identity \eqref{eq:xidef}, and hence that $a(z;n)\leq 1$ for every $n\geq 0$ by the inequality \eqref{eq:abound}. Since this bound holds for all $z<z_c,$ it also holds for $z=z_c$ by left continuity.
\end{proof}

A similar analysis shows that $(b_n)_{n\geq0}$ is supermultiplicative and that 
\begin{equation}
\label{eq:bridges}
b_n \leq \mu_c^n
\end{equation}
for every $n\geq 0$ \cite[eq.\ 1.2.17]{MR2986656}.




\begin{proof}[Proof of Proposition \ref{prop:weakHW}] 
We have the trivial inequality
\begin{equation}
\label{eq:trivial}a((1-\eps)z_c;n) \leq (1-\eps)^n a(z_c;n) \leq (1-\eps)^n.\end{equation}
It follows that $B((1-\eps)z_c) \leq \eps^{-1}$, and hence that
\[
\chi((1-\eps)z_c) \leq \frac{1}{z}\exp\left[2\eps^{-1} -2\right].
\]
To conclude, we apply the trivial inequality
\begin{equation}
\label{eq:trivial2}
z_c^n c_n \leq (1-\eps)^{-n}\chi((1-\eps)z_c) 
\end{equation}
with $\eps = (n/2)^{-1/2}$ to deduce that
\begin{align*}
z_c^{n+1} c_n &\leq 
\exp\left[\sqrt{2n} -2\right]\Bigl(1-\sqrt{2/n}\Bigr)^{-n-1} =  \exp\left[\sqrt{8n}+ O(\sqrt{1/n})\right]
\end{align*}
as $n\to\infty$, where the equality on the right-hand side follows by calculus (the $-2$ has not been forgotten).
\end{proof}

\subsection{Proof of the improvement}

The main idea behind \cref{thm:main,thm:quant} is that \cref{thm:DCH} allows us to improve upon the trivial inequality \eqref{eq:trivial}. This improvement is encapsulated in the following lemma.


\begin{lemma}$
\xi((1-\eps)z_c) \geq -\psi(\eps)\log(1-\eps)$ 
for every $\eps>0$. 
\end{lemma}

\begin{proof}
It suffices to prove that
\begin{equation}
\label{eq:epslambda}
\xi((1-\eps)z_c) \geq \min\bigl\{ -\lambda\log(1-\eps), -\log(1-\eps)+ \Phi(\lambda^{-1})\bigr\}\end{equation}
for every $\eps>0$ and $\lambda\geq 1$, since the result then follows by optimizing over $\lambda$.
Splitting the walks contributing to $a((1-\eps)z_c;n)$ according to whether they have length more than $\lambda n$ or not yields that
\begin{multline*}
a((1-\eps)z_c;n) \leq \sum_{m\geq \lambda n} \sum_{\omega \in \Omega} z_c^m(1-\eps)^m
\mathbbm{1}\bigl[\omega: 0 \to L_n \text{ a length $m$ SAB} \bigr]\\
+  \sum_{m=n}^{\lambda n} \sum_{\omega \in \Omega} z_c^m(1-\eps)^m\mathbbm{1}\bigl[\omega: 0 \to L_n \text{ a length $m$ SAB} \bigr]
\end{multline*}
and hence that
\begin{multline*}
a((1-\eps)z_c;n) \leq (1-\eps)^{\lambda n}\sum_{m\geq \lambda n} \sum_{\omega \in \Omega} z_c^m
\mathbbm{1}\bigl[\omega: 0 \to L_n \text{ a length $m$ SAB} \bigr]\\
+  (1-\eps)^n \sum_{m=n}^{\lambda n} \sum_{\omega \in \Omega} z_c^m \mathbbm{1}\bigl[\omega: 0 \to L_n \text{ a length $m$ SAB} \bigr].
\end{multline*}
This implies that
\begin{align*}
a((1-\eps)z_c;n)
&\leq (1-\eps)^{\lambda n} a(z;n) +  (1-\eps)^n\sum_{m = n}^{\lambda n} z_c^m b_m \exp\left[-\phi\left(\frac{n}{m}\right)m\right]\\
&\leq (1-\eps)^{\lambda n} + \lambda n (1-\eps)^n \exp\left[-\Phi(\lambda^{-1})n\right]
\end{align*}
for every $n\geq 0$, where \eqref{eq:abound}, Lemma \ref{lem:xizc}, \eqref{eq:bridges} and the definition of $\Phi$ were used in the second inequality. The claimed inequality \eqref{eq:epslambda} follows. 
\end{proof}

\begin{proof}[Proof of \cref{thm:quant}]
Applying \eqref{eq:epslambda} and \eqref{eq:abound} yields the estimate
\[
B((1-\eps)z_c)  = \sum_{n\geq 0} a((1-\eps)z_c;n) \leq \sum_{n\geq 0} e^{\psi(\eps)\log(1-\eps)n} = \left[1-(1-\eps)^{\psi(\eps)}\right]^{-1} 
\]
We deduce from Proposition \ref{prop:M&S} that
\[\chi((1-\eps)z_c) \leq \frac{1}{(1-\eps)z_c}\exp\left[ 2\left[1-(1-\eps)^{\psi(\eps)}\right]^{-1}-2\right].\]
The claim now follows by applying the trivial inequality \eqref{eq:trivial2} as in the proof of Proposition \ref{prop:weakHW}.
\end{proof}

\begin{proof}[Proof of \cref{thm:main}]
Let $\phi$ be as in \cref{thm:DCH}.  It suffices to prove that $\Psi(n)=o(n^{1/2})$. Since $\phi(\eps)$ is increasing and $\phi(\eps)>0$ for all $0<\eps\leq 1$, we have that $\Phi(\eps)>0$ for all $0<\eps\leq 1$, and hence that $\psi(\eps)\to\infty$ as $\eps\to 0$. Fix $M \geq 1$ and $\delta>0$. Then there exists $0<\eps_0\leq 1$ such that $\psi(\eps)\geq M$ for every $0<\eps \leq \eps_0$.
 Setting $\eps=\delta n^{-1/2}$ yields that
\[
\Psi(n) \leq 2
\left[1-(1-\delta n^{-1/2} )^{M}\right]^{-1} 
- (n+1) \log(1-\delta n^{-1/2}) 
\]
for every sufficiently large $n$. It follows by calculus that
\[
\Psi(n) \leq \left(\delta + \frac{2}{\delta M}\right) n^{1/2} + o(n^{1/2})
\]
as $n\to \infty$. The result follows since $M\geq 1$ and $\delta>0$ were arbitrary (e.g.\ by taking $M=\delta^{-2}$ and sending $\delta\to0$).
\end{proof}

\begin{proof}[Proof of Corollary \ref{cor:poly}]
Suppose that $\phi(\eps)\geq C\eps^{\nu}$ for some $C>0$ and $\nu>1$, and that $\phi$ satisfies \eqref{eq:phidef}. Then $\Phi(\eps)\geq C\eps^{\nu-1}$ and hence
\[
\psi(\eps) \geq \sup\left\{\lambda \geq 1 : \eps \leq 1- \exp\left[-C \frac{\lambda^{1-\nu}}{\lambda-1} \right]\right\}.
\]
A straightforward analysis then yields that 
\[
\psi(\eps) \geq C' \eps^{-1/\nu}
\]
for some $C'>0$ and every $\eps>0$ sufficiently small. Let $\alpha>0$. Then
\begin{align*}
\Psi(n) \leq 2\left[1-(1-n^{-\alpha})^{C' n^{\alpha/\nu}}\right]^{-1} - (n+1)\log(1-n^{-\alpha})
\end{align*}
for every $n$ sufficiently large. It follows by calculus that
\[
\Psi(n) = O\left( n^{\alpha(\nu-1)/\nu} + n^{1-\alpha}\right)
\]
as $n\to\infty$. Optimizing by taking $\alpha = \nu/(2\nu-1)$ and applying \cref{thm:quant} yields the claim.
\end{proof}


\begin{remark}
\label{remark}
We now explain why values of $\nu\leq 1$ are not possible in Corollary \ref{cor:poly}. Suppose that $\phi$ satisfies \eqref{eq:phidef}. An inequality of Madras and Slade \cite[eq.\ 3.1.14]{MR2986656} (which is a trivial consequence of Proposition \ref{prop:M&S}) states that
\[
B((1-\eps)z_c) \geq 1+\frac{1}{2}\log[(1-\eps)z_c]+ \frac{1}{2}\log \eps^{-1},
\]
and it follows that
\[
\left[1-(1-\eps)^{\psi(\eps)}\right]^{-1} \geq 1+\frac{1}{2}\log[(1-\eps)z_c]+ \frac{1}{2}\log \eps^{-1}.
\]
This rearranges to give
\[
\psi(\eps) \leq \frac{2+o(1)}{-\eps \log \eps}
\quad \text{ and hence } \quad
\phi\left(\frac{-\eps \log \eps}{2+o(1)}\right) \leq  (1+o(1))\eps 
\]
 as $\eps\to 0$. Probabilistically, this yields a lower bound on the asymptotic probability that an $n$-step SAB has height at least $\eps n$. 
\end{remark}

\subsection*{Acknowledgments} The author was supported by internships at Microsoft Research and a Microsoft Research PhD Fellowship. We thank Omer Angel, Hugo Duminil-Copin, Tyler Helmuth and Gordon Slade for comments on an earlier draft. Finally, we thank the anonymous referee for catching several errors in the preprint. 

\footnotesize{
  \bibliographystyle{abbrv}
  \bibliography{unimodularthesis}
  }

\end{document}